\documentclass[a4paper, 10pt]{article}
\usepackage[T2A]{fontenc}
\usepackage[english]{babel}
\usepackage{amsmath, amsthm, amsfonts, amssymb, accents}
\usepackage{graphicx,dsfont,color}

\usepackage{srcltx}

\numberwithin{equation}{section}
\theoremstyle{plain}
\newtheorem{theorem}{Theorem}[section]
\newtheorem{lemma}{Lemma}[section]
\theoremstyle{definition}

\numberwithin{equation}{section}

\def\tht{\theta}

\def\om{\omega}
\def\e{\varepsilon}
\def\g{\gamma}

\def\l{\lambda}
\def\p{\partial}
\def\D{\Delta}

\def\a{\alpha}
\def\b{\beta}

\def\z{\zeta}

\def\vp{\varphi}

\def\H{W_2}
\def\Ho{\mathring{H}_2}

\def\iu{\mathrm{i}}
\def\di{\,d}

\def\Op{\mathcal{H}}

\def\tauo{\tfrac{1}{2}}

\def\Beta{\mathrm{B}}

\def\cL{\mathcal{L}}
\def\cS{\mathcal{S}}

\def\H{W_2}
\def\Ho{\mathring{W}_2}

\def\Hoper{\mathring{W}_{2,per}}

\DeclareMathOperator{\spec}{\sigma}

\begin{document}

\title{Bethe-Sommerfeld conjecture for periodic Schr\"odinger operators in strip}
\author{D.I. Borisov}

\date{\empty}

\maketitle

\allowdisplaybreaks

\begin{abstract}
We consider the Dirichlet Laplacian in a straight planar strip perturbed by a bounded periodic symmetric operator. We prove the classical Bethe-Sommerfeld conjecture for this operator, namely, that this operator has finitely many gaps in its spectrum provided a certain special function written as a series satisfies some lower bound. We show that this is indeed the case if the ratio of the  period and the width of strip is less than a certain explicit  number, which is approximately equal to 0.10121. We also find explicitly the point in the spectrum, above which there is no internal gaps. We then study the case of a sufficiently small period and we prove that in such case the considered operator has no internal gaps in the spectrum. The conditions ensuring the absence are written as certain explicit inequalities.
\end{abstract}

\footnotetext[1]{Institute of Mathematics, Ufa Federal Research Center, Russian Academy of Sciences, Ufa, Russia, email: borisovdi@yandex.ru
\\
Bashkir State University, Ufa, Russia
\\
University of Hradec Kr\'alov\'e,  Hradec Kr\'alov\'e, Czech Republic
}

\section{Introduction}

The classical Bethe-Sommerfeld conjecture says that a multi-dimensional periodic differential operator has finitely many gaps in its spectrum. This conjecture was proved for a wide class of operators in multi-dimensional spaces.  The case of Schr\"odinger operator with a periodic potential or, more generally, with a bounded periodic symmetric operator, was studied in \cite{DT},
\cite{SS05}, \cite{P08}, \cite{HM98}, \cite{SS05-2}, \cite{S85}, \cite{Vel}, \cite{Kar} and the Bethe-Sommerfeld conjecture was proved under various conditions for the potential and the bounded periodic symmetric operator including various cases of unbounded potentials.

In \cite{K04}, \cite{M97}, this conjecture was proved for the magnetic Schr\"odinger operator. Papers  \cite{PS01}, \cite{BP09}, \cite{PS10} were devoted to proving the Bethe-Sommerfeld conjecture for polyharmonic operators perturbed by a pseudodifferential operator of a lower order obeying certain conditions.

Apart of operators in multi-dimensional spaces, the Bethe-Sommerfeld conjecture can be formulated also for differential operators in periodic domains. The examples of such domains are strips, cylinders or layers. Here the simplest model is the periodic Schr\"odinger operator in a two-dimensional planar strip. Such model was studied in PhD thesis \cite{Be}. Assuming that
\begin{equation}\label{1.1}
\frac{T}{d}<\frac{1}{16},
\end{equation}
where $d$ was the width of the strip and $2T$ was the period, it was proved in \cite{Be} that the considered operator has finitely many gaps in the spectrum. We also note that this result appeared only in the cited PhD thesis and was not published  as a usual paper in a journal.

One more example of differential operators on periodic domains are the operators on periodic graphs. Here the situation changes substantially and the operators on the graphs can have infinitely many gaps, or finitely many gaps  or no gaps at all. More details on the Bethe-Sommerfeld conjecture for operators on periodic graphs can be found, for instance, in recent work \cite{ET} devoted to finding examples of quantum graphs with finitely many gaps in the spectra.

The Bethe-Sommerfeld conjecture can be interpreted as the absence of the gaps above some point, that is, in the higher part of the spectrum. This suggests another problem on finding the periodic operators having no gaps at all; such problem can be called  a strong Bethe-Sommerfeld conjecture. This issue was studied, for instance, in \cite{DT} and \cite{S85}. It was found that for the periodic Schr\"odinger operators in the multi-dimensional space this is true provided the potential is small enough, see Remark in \cite{DT} and Theorems~15.2~and~15.6 in \cite[Ch. I\!I\!I, Sect. 15]{S85}. By a simple rescaling, this result can be also reformulated as follows: the periodic Schr\"odinger operator in entire space has no gaps if the period is small enough.

The aforementioned results on the strong Bethe-Sommerfeld conjecture motivated very recent studies on periodic operators with a small period in \cite{UMJ18}, \cite{TMF18}, \cite{DFDE17}, \cite{IMM18}.
The considered operators were a periodic Schr\"odinger operator \cite{UMJ18}, a periodic magnetic Schr\"odinger operator \cite{DFDE17}, the Laplacian with frequently alternating boundary conditions \cite{TMF18} and the Laplacian with a periodic delta interaction \cite{IMM18}. The main result of the cited works was as follows: for a sufficiently small period, as $T<T_0$, the considered operators has no internal spectral gaps at least till certain point $\l_T$ in the spectrum. The upper bound $T_0$ for the period ensuring this result was found explicitly, as a particular number. The point $\l_T$ was also found explicitly as a rather simple function of $T$. It was shown that $\l_T$ behaved as $O(T^{-6})$ as $T$ goes to zero. We stress that this result does not state the absence of the gaps in the entire spectrum but only in its lower part. At the same time, we succeeded to consider more complicated operators and not only the periodic Schr\"odinger operator. Here it is important to stress that the approach used in \cite{Be} is rather limited and it can not be extended to the operators with stronger perturbations like in \cite{TMF18}, \cite{DFDE17}, \cite{IMM18}. The technique in the latter works, namely, the key estimates, were based on different ideas in comparison with that in \cite{Be}.

In the present work we study the same model as in \cite{Be}, namely, the Dirichlet Laplacian in a strip perturbed by a bounded periodic symmetric operator. We again study the internal gaps in the spectrum but in greater details.   Our first result is the proof of the classical Bethe-Sommerfeld conjecture under weaker conditions. Namely, we show that it is true provided
\begin{equation}\label{1.2}
\frac{T}{d}<\xi_0\approx  0.10121
\end{equation}
and this condition is better than (\ref{1.1}). For other values of $\frac{T}{d}$, the classical Bethe-Sommerfeld conjecture holds if  a certain special function written explicitly as a series satisfies certain lower bound, see (\ref{2.4}), (\ref{2.3}). Here we also find  explicitly a point in the spectrum, above which there is surely no gaps.

We also study the case of a small period. Here we prove that   if condition (\ref{1.2}) is satisfied and the numerical range of the perturbation is not too wide, the considered operator has no internal gaps in the spectrum. The conditions for the perturbation are  explicit and rather simple, see (\ref{2.9}), (\ref{2.10}). In particular, these conditions imply the following statement: varying the period of a potential and keeping its oscillation uniformly bounded, for sufficiently small periods the corresponding periodic Schr\"odinger operator in the strip has no internal gaps. This result fits very well what was said above about periodic Schr\"odinger operators in multi-dimensional space with small periods.

The approach we use follows the same lines as in
\cite{Be}, namely, it is based on the ideas from \cite{DT}.  But while proving the key estimates for the Fourier coefficients of the counting function, we succeeded to do this in a shorter and simpler way tracking at the same time all the constants explicitly. In the proof of the strong Bethe-Sommerfeld conjecture we also employ the approach developed in \cite{UMJ18}, \cite{TMF18}, \cite{DFDE17},  \cite{IMM18}.

\section{Problem and main results}

Let $x=(x_1,x_2)$ be Cartesian coordinates in $\mathds{R}^2$, $\Pi:=\{x:\, 0<x_2<d\}$ be an infinite horizontal strip of a width $d>0$, and $\square:=\{x:\, |x_1|<T, 0<x_2<d\}$ be a periodicity cell, where $T>0$ is a constant. By $\cL_0$ we denote a bounded symmetric operator in $L_2(\square)$ and $\cS(n)$ stands for the translation operator in $L_2(\Pi)$ acting as $(\cS(n)u)(x)=u(x_1-2 T n, x_2)$. By means of the operators $\cL_0$ and $\cS(n)$ we introduce one more operator in $L_2(\Pi)$:
\begin{equation*}
\cL u=\cS(-n)\cL_0\cS(n)u\quad\text{on}\quad \square_n,\quad n\in\mathds{Z},
\end{equation*}
where $\square_n:=\{x:\,(x_1-2Tn,x_2)\in\square\}$. This definition of the operator $\cL$ can be explained as follows: the restriction on $\square_n$ of a function $u\in L_2(\Pi)$ belongs to $L_2(\square)$. Identifying then the spaces $L_2(\square)$ and $L_2(\square_n)$, we apply the operator $\cL_0$ to the restriction $u\big|_{\square_n}$ and the result is translated to the cell $\square_n$. This is the action of the operator $\cL$ on $u$ on the cell $\square_n$.

The operator $\cL$ is bounded, symmetric and periodic. The latter is understood in the sense of the identity
\begin{equation*}
\cS(m)\cL=\cL\cS(m)\quad\text{for each}\quad m\in\mathds{Z}.
\end{equation*}

The main object of our study is the periodic operator
\begin{equation*}
\Op:=-\D+\cL\quad\text{in}\quad L_2(\Pi)
\end{equation*}
subject to the Dirichlet condition. The domain of this operator is the Sobolev space $\Ho^2(\Pi)$ consisting of the functions in $\H^2(\Pi)$ with the zero trace on $\p\Pi$. The operator $\Op$ is self-adjoint.

We denote
\begin{equation}\label{2.5}
\om_-:= \inf\limits_{\substack{u\in L_2(\square) \\ u\ne 0}}  \frac{(\cL u,u)_{L_2(\square)}}{\|u\|_{L_2(\square)}^2},\quad
\om_+:= \sup\limits_{\substack{u\in L_2(\square) \\ u\ne 0}}  \frac{(\cL u,u)_{L_2(\square)}}{\|u\|_{L_2(\square)}^2},\quad
\om_\cL:=\om_+ - \om_-,
\end{equation}
and
\begin{equation*}
\xi:=\frac{T}{d}.
\end{equation*}
Given $\xi>0$, for  $\ell>0$, $p\in\mathds{N}$ we introduce the function:
\begin{equation}\label{2.4}
\vp_p(\ell):=\frac{1}{\pi\xi}  \sum\limits_{k\in\mathds{Z}}
\frac{\sin\left( 2\pi  \ell^{\frac{1}{2}} \sqrt{\frac{k^2}{\xi^2}+p^2} -\frac{\pi}{4}\right)}{\left(\frac{k^2}{\xi^2}+p^2\right)^\frac{3}{4}}.
\end{equation}

\begin{theorem}\label{th2.1}
Assume that for a given $\xi$ there exists three constants $c_0=c_0(\xi)>0$, $\ell_0=\ell_0(\xi)\geqslant 1$ and $\g=\g(\xi)<\tfrac{1}{4}$ such that for $\ell\geqslant \ell_0$ the inequality holds:
\begin{equation}\label{2.3}
\sup\limits_{p\in\mathds{N}}|\vp_p(\xi,\ell)|\geqslant c_0 \ell^{-\g}.
\end{equation}
Then, for this $\xi$, the  spectrum of the operator $\Op$ has finitely many internal gaps. Moreover, there are no internal gaps in the half-line $[\ell_{1},+\infty)$, where
\begin{equation}
\begin{aligned}
\ell_{1}:=&\frac{\pi^2}{T^2}\max\Bigg\{\ell_0,\left(\frac{4\sqrt{2}\pi+6}{3\pi c_0}\right)^\frac{4}{1-4\g},\left(\frac{1}{8\pi^2\xi c_0}\sqrt{9+\frac{25}{1024\pi^2}}\right)^\frac{2}{1-2\g},
\\
 &\hphantom{\frac{\pi^2}{T^2}\max\Bigg\{} \left(\frac{T}{4 \pi c_0 \xi}\om_\cL+\frac{1}{2c_0} + \frac{3\xi T}{4\pi c_0}\right)^{\frac{4}{1-4\g}}\Bigg\}
+\om_-.
\end{aligned}
\label{2.11}
\end{equation}
\end{theorem}

Our next main result states that condition (\ref{2.3}) holds for sufficiently small $\xi$.

\begin{theorem}\label{th2.3}
Let
\begin{equation}
\label{2.8}
  \xi<\xi_0,
\end{equation}
where
\begin{align*}
&\xi_0:=\left(\frac{c_1}{2\z\left(\frac{3}{2}\right)}\right)^\frac{2}{3}\approx  0.10121,
\\
& c_1:= \frac{c_2}{\sqrt{c_2^2+1}},
\qquad
 c_2:=\frac{(78\sqrt{3}+54\sqrt{11})^\frac{1}{3}-\sqrt{3}}{9} - \frac{8}{3(78\sqrt{3}+54\sqrt{11})^\frac{1}{3}},
\end{align*}
where $\z(t)$ is the Riemann zeta function.
Then condition (\ref{2.3}) holds with
$$
\ell_0(\xi)=1,\qquad \g=0,\qquad c_0=\frac{c_1-2\z\left(\frac{3}{2}\right)\xi^\frac{3}{2}}{\pi\xi}
$$
and the statement of Theorem~\ref{th2.1} is true.
\end{theorem}

In the next theorem we prove the strong Bethe-Sommerfeld conjecture.

\begin{theorem}\label{th2.2}
Assume that condition (\ref{2.8}) holds and
\begin{align}
  &0\leqslant \frac{T^2}{\pi^2}\om_\cL<\frac{((A(\xi)-\xi)^2+1)^2}{4}-A^2(\xi), \quad A(\xi):=\frac{\sqrt{3+4\xi^2}+\xi}{3},\label{2.9}
  \\
  &
  \begin{aligned}
  0\leqslant & c_1-2\z\left(\frac{3}{2}\right)\xi^\frac{3}{2}
  -\frac{(3+2\sqrt{2})\pi+3}{6}\xi-\frac{3\pi}{4}\xi^2
- \frac{\xi}{32\pi} \sqrt{9+\frac{25}{1024\pi^2}}   -\frac{T\om_\cL}{4}.
  \end{aligned}\label{2.10}
\end{align}
 Then the  spectrum of the operator $\Op$ has no internal gaps.
\end{theorem}

Let us discuss briefly the main results. The first theorem states the classical Bethe-Sommerfeld conjecture. Namely, provided estimate (\ref{2.11}) holds, the considered operator has finitely many gaps in its spectrum and surely there are no gaps above the point $\ell_1$.
The natural question is whether estimate (\ref{2.11}) is true or not. Theorem~\ref{th2.3} says that provided $\xi$ is not too big, namely, if $\xi$ obeys (\ref{2.8}),   estimate (\ref{2.11}) is true and the classical Bethe-Sommerfeld conjecture holds. We stress that condition~(\ref{2.11}) is better than similar condition~(\ref{1.1}) in~\cite{Be} since $\tfrac{1}{16}=0.0625$ is less than $\xi_0$. We failed to prove estimate (\ref{2.11}) for other values of $\xi$ but numerical tests show that this estimate is  very likely true for all values of $\xi$. In  Section~6 we discuss the functions $\vp_p(\ell)$ and condition (\ref{2.3}) in more details.

Theorem~\ref{th2.2} is devoted to the case of a small period. Here we prove the absence of the internal gaps    in the spectrum provided conditions~(\ref{2.8}),~(\ref{2.9}),~(\ref{2.10}) are satisfied. And as we see easily, these conditions hold for a sufficiently small period $T$ assuming that the width $d$ and the oscillation  $\om_\cL$ are fixed. For instance, this implies that given a fixed bounded potential $V(x_1,x_2)$, which is $2\pi$-periodic in  $x_1$, the  Schr\"odinger operator $-\D+V(\tfrac{x_1}{\e},x_2)$ in the strip $\Pi$ subject to the Dirichlet boundary condition has no internal gaps provided $\e$ is small enough. In view of this result and the aforementioned results in \cite{DT}, \cite{S85}, \cite{TMF18}, \cite{UMJ18}, \cite{IMM18}, \cite{DFDE17}, we could formulate a strong Bethe-Sommerfeld conjecture: \emph{multi-dimensional periodic differential operators, for which the classical Bethe-Sommerfeld conjecture holds, have no internal gaps in their spectra if the period is small enough}.

In conclusion we stress that our technique and results can be also extended to the case of Neumann or Robin boundary condition (with a constant coefficient). Of course, for other boundary conditions all the constants in Theorems~\ref{th2.1},~\ref{th2.3},~\ref{th2.2} are different.

\section{Counting functions}

This section is devoted to the preliminary notations and statements used then the proofs of Theorems~\ref{th2.1},~\ref{th2.3},~\ref{th2.2}.

Since the operator $\Op$ is periodic, its spectrum has a band structure and it can be described in terms of the band functions. In order to do this, we first define the operator
\begin{equation*}
\Op(\tau):=\left(\iu\frac{\p\ }{\p x_1}+\frac{\pi\tau}{T}\right)^2 -\frac{\p^2}{\p x_2^2} + e^{\iu\frac{\pi\tau}{T} x_1} \cL e^{-\iu \frac{\pi\tau}{T} x_1},\quad \tau\in\left(-\tauo,\tauo\right],
\end{equation*}
in $L_2(\square)$ subject to the Dirichlet condition on $\p\square\cap\p\Pi$ and to the periodic boundary conditions on the lateral boundaries of $\square$. The domain of this operator is the space $\Hoper^2(\square)$ consisting of the functions in $\H^2(\square)$ satisfying the Dirichlet condition on $\p\square\cap\p\Pi$ and the periodic conditions on the lateral boundaries of $\square$.

The operator $\Op(\tau)$ is self-adjoint and has a compact resolvent. The spectrum of the operator $\Op(\tau)$ consists of countably many discrete eigenvalues. These eigenvalues are taken in the ascending order counting the multiplicities and are denoted by $E_k(\tau)$, $k\in\mathds{N}$. By $E_k^0(\tau)$ we denote the same eigenvalues in the case $\cL=0$, that is, they are associated with the Dirichlet Laplacian in $\Pi$. The latter operator is denoted by $\Op_0$ and the associated operator on the periodicity cell $\square$ is $\Op_0(\tau)$.

The well-known formulae for the spectra of the operators $\Op$ and $\Op_0$ are
\begin{equation*}
\spec(\Op)=\bigcup\limits_{k\in\mathds{Z}} \left\{E_k(\tau): \tau\in\left(-\tauo,\tauo\right]\right\},\quad \spec(\Op_0)=\bigcup\limits_{k\in\mathds{Z}} \left\{E_k^0(\tau): \tau\in\left(-\tauo,\tauo\right]\right\}.
\end{equation*}

By $N_0(\ell,\tau)$ we denote the rescaled counting function of the operator $\Op(\tau)$ in the case $\cL=0$:
\begin{equation}\label{3.3}
N_0(\ell,\tau)=\#\left\{E_k^0(\tau):\, E_k^0(\tau)\leqslant \frac{\pi^2 \ell}{T^2}\right\}.
\end{equation}

Since the function $N_0(\ell,\tau)$ is associated with the Dirichlet Laplacian in $\Pi$, we can calculate explicitly the eigenvalues of $\Op^0(\tau)$:
\begin{equation}\label{3.6}
\big\{E_k^0(\tau),\ k\in\mathds{N}\big\}=\bigg\{\frac{\pi^2}{T^2}(\tau+n)^2+\frac{\pi^2 m^2}{d^2},\ n\in\mathds{Z},\ m\in\mathds{N}\bigg\}.
\end{equation}
The eigenvalues in the right hand side correspond to the eigenfunctions $e^{-\iu \frac{\pi n}{T}x_1}\sin\frac{\pi m}{d}x_2$ and they do not follow the ascending order. This is why we write (\ref{3.6}) as the identity for two sets of the eigenvalues.

The counting function $N_0(\ell,\tau)$ can be written as
\begin{equation}\label{3.7}
N_0(\ell,\tau)=\sum\limits_{\substack{n\in\mathds{Z},\, m\in\mathds{N} \\
(n+\tau)^2+\xi^2 m^2\leqslant \ell}} 1
= \sum\limits_{\substack{n\in\mathds{Z}_+,\, m\in\mathds{N}
\\
(n+\tau)^2+\xi^2 m^2\leqslant \ell}} 1 +
\sum\limits_{\substack{n\in\mathds{Z}_+,\, m\in\mathds{N}
\\
(n+1-\tau)^2+\xi^2 m^2\leqslant \ell}} 1.
\end{equation}

The definition of the function $N_0$ implies immediately that this function is even in $\tau\in\left[-\tauo,\tauo\right]$. By $a_p$ we denote the Fourier coefficients of this function:
\begin{equation}\label{3.8}
a_0(\ell):=\int\limits_{-\frac{1}{2}}^{\frac{1}{2}} N_0(\ell,\tau)\di\tau,
\qquad
a_p(\ell):=
\int\limits_{-\tauo}^{\tauo}  N_0(\ell,\tau)\cos(2\pi p\tau)\di\tau,\quad p\in\mathds{N};
\end{equation}
the Fourier series for $N_0(\ell,\tau)$ reads as
\begin{equation*}
N_0(\ell,\tau)=a_0(\ell)+2\sum\limits_{p=1}^{\infty} a_p(\ell)\cos2\pi p\tau.
\end{equation*}

The functions $E_k^0(\tau)$ satisfy the estimate $E_k^0(\tau)\geqslant \frac{\pi^2}{d^2}$ and therefore, the counting function $N_0(\ell,\tau)$ is non-zero only for
\begin{equation*}
 \ell\geqslant \xi^2.
\end{equation*}
In what follows we assume that this inequality is satisfied.

\subsection{Coefficient $a_0(\ell)$.} In this subsection we calculate and estimate the coefficient $a_0(\ell)$.

By $\lfloor\cdot\rfloor$ we denote the integer part of a number, while $\lceil\cdot\rceil$ stands for the fractional part.
By straightforward calculations we get:
\begin{align*}
a_0(\ell)=&\sum\limits_{\substack{n\in\mathds{Z}_+,\, m\in\mathds{N} \\ (n+\tauo)^2+\xi^2 m^2\leqslant \ell}} 1 +  \sum\limits_{\substack{n\in\mathds{Z}_+,\, m\in\mathds{N} \\ (n+1)^2+\xi^2 m^2\leqslant \ell}} 1 + 2\int\limits_{0}^{\tauo} \sum\limits_{\substack{
n\in\mathds{Z}_+,\, m\in\mathds{N}
\\
(n+\tau)^2+\xi^2 m^2\leqslant \ell
\\
n^2+\xi^2 m^2\leqslant \ell<\left(n+\tauo\right)^2+\xi^2 m^2
}} \di \tau
\\
&+ 2\int\limits_{0}^{\tauo} \sum\limits_{\substack{
n\in\mathds{Z}_+,\, m\in\mathds{N}
\\
(n+1-\tau)^2+\xi^2 m^2\leqslant \ell
\\
\left(n+\tauo\right)^2+\xi^2 m^2\leqslant \ell<(n+1)^2+\xi^2 m^2
}} \di \tau
\\
=&\sum\limits_{m=1}^{\left\lfloor\frac{\ell^{\frac{1}{2}}}{\xi}\right\rfloor} \left\lfloor\sqrt{\ell-\xi^2 m^2}+\tauo\right\rfloor + \sum\limits_{m=1}^{\left\lfloor\frac{\ell^{\frac{1}{2}}}{\xi}\right\rfloor} \left\lfloor\sqrt{\ell-\xi^2 m^2}\right\rfloor
\\
&+2
 \sum\limits_{\substack{
n\in\mathds{Z}_+,\, m\in\mathds{N}
\\
n^2+\xi^2 m^2\leqslant \ell<\left(n+\tauo\right)^2+\xi^2 m^2
}} \left(\sqrt{\ell-\xi^2 m^2}-n\right)
\\
&+  2 \sum\limits_{\substack{
n\in\mathds{Z}_+,\, m\in\mathds{N}
\\
\left(n+\tauo\right)^2+\xi^2 m^2\leqslant \ell<(n+1)^2+\xi^2 m^2
}} \left(\sqrt{\ell-\xi^2 m^2}-n-\tauo\right).
\end{align*}
Hence,
\begin{align*}
a_0(\ell)=&  \sum\limits_{\substack{m=1 \\ \left\lceil\sqrt{\ell-\xi^2 m^2}\right\rceil<\tauo
}}^{\left\lfloor\frac{\ell^{\frac{1}{2}}}{\xi}\right\rfloor} 2\left\lfloor\sqrt{\ell-\xi^2 m^2}\right\rfloor +
\sum\limits_{\substack{m=1 \\ \left\lceil\sqrt{\ell-\xi^2 m^2}\right\rceil<\tauo
}}^{\left\lfloor\frac{\ell^{\frac{1}{2}}}{\xi}\right\rfloor} 2\left\lceil\sqrt{\ell-\xi^2 m^2}\right\rceil
\\
&+ \sum\limits_{\substack{m=1 \\ \left\lceil\sqrt{\ell-\xi^2 m^2}\right\rceil\geqslant\tauo
}}^{\left\lfloor\frac{\ell^{\frac{1}{2}}}{\xi}\right\rfloor} \left(2\left\lfloor\sqrt{\ell-\xi^2 m^2}\right\rfloor+1\right)
+ \sum\limits_{\substack{m=1 \\ \left\lceil\sqrt{\ell-\xi^2 m^2}\right\rceil\geqslant \tauo
}}^{\left\lfloor\frac{\ell^{\frac{1}{2}}}{\xi}\right\rfloor} \left(2\left\lceil\sqrt{\ell-\xi^2 m^2}\right\rceil-1\right)
\end{align*}
and therefore,
\begin{equation}\label{3.12}
a_0(\ell)=2\sum\limits_{m=1}^{\big\lfloor\frac{\ell^{\frac{1}{2}}}{\xi}\big\rfloor}
\sqrt{\ell-\xi^2 m^2}.
\end{equation}

\begin{lemma}\label{lm3.2}
The function $a_0(\ell)$ is monotonically increasing and for each
$\xi^2\leqslant \ell\leqslant \tilde{\ell}$ the estimate holds:
\begin{equation*}
a_0(\tilde{\ell})-a_0(\ell)\leqslant \frac{\pi}{2\xi}(\tilde{\ell}-\ell)+\sqrt{\tilde{\ell}-\ell}.
\end{equation*}
\end{lemma}

\begin{proof}
By formula (\ref{3.12}) we have:
\begin{align*}
a_0(\tilde{\ell})-a_0(\ell)=& 2\sum\limits_{m=1}^{\Big\lfloor\frac{\tilde{\ell}^{\frac{1}{2}}}{\xi}\Big\rfloor}
\sqrt{\tilde{\ell}-\xi^2 m^2} - 2\sum\limits_{m=1}^{\Big\lfloor\frac{\ell^{\frac{1}{2}}}{\xi}\Big\rfloor}
\sqrt{\ell-\xi^2 m^2}
\\
=  & 2\sum\limits_{m=1}^{\Big\lfloor\frac{\ell^{\frac{1}{2}}}{\xi}\Big\rfloor}
\Big(\sqrt{\tilde{\ell}-\xi^2 m^2}-\sqrt{\ell-\xi^2 m^2}\Big) + 2\sum\limits_{m=\Big\lfloor\frac{\ell^{\frac{1}{2}}}{\xi}\Big\rfloor +1}^{\Big\lfloor\frac{\tilde{\ell}^{\frac{1}{2}}}{\xi}\Big\rfloor}
\sqrt{\tilde{\ell}-\xi^2 m^2}
\\
\\
=  & 2\sum\limits_{m=1}^{\Big\lfloor\frac{\ell^{\frac{1}{2}}}{\xi}\Big\rfloor}
\frac{\tilde{\ell}-\ell}{\sqrt{\tilde{\ell}-\xi^2 m^2}+\sqrt{\ell-\xi^2 m^2}}  + 2\sum\limits_{m=\Big\lfloor\frac{\ell^{\frac{1}{2}}}{\xi}\Big\rfloor +1}^{\Big\lfloor\frac{\tilde{\ell}^{\frac{1}{2}}}{\xi}\Big\rfloor}
\sqrt{\tilde{\ell}-\xi^2 m^2}.
\end{align*}
Since the function
\begin{equation*}
t\mapsto  \frac{1}{\sqrt{\tilde{\ell}-\xi^2 t^2}+
\sqrt{\ell-\xi^2 t^2}}
\end{equation*}
is monotonically increasing as $t\in\big[0,\tfrac{\ell^\frac{1}{2}}{\xi}\big]$, and the function $t\mapsto \sqrt{\tilde{\ell}-\xi^2 t^2}$ is monotonically
decreasing as $t\in\big[\tfrac{\ell^\frac{1}{2}}{\xi},\tfrac{\tilde{\ell}^\frac{1}{2}}{\xi}\big]$, we can continue estimating as follows:
\begin{align*}
a_0(\tilde{\ell})-a_0(\ell)\leqslant& 2(\tilde{\ell}-\ell) \int\limits_{0}^{\Big\lfloor\frac{\ell^{\frac{1}{2}}}{\xi}\Big\rfloor} \frac{\di t}{\sqrt{\tilde{\ell}-\xi^2 t^2}+
\sqrt{\ell-\xi^2 t^2}}
 +\frac{2(\tilde{\ell}-\ell)}{\sqrt{\tilde{\ell}-\xi^2 \lfloor\frac{\ell^\frac{1}{2}}{\xi}\rfloor^2}+
\sqrt{\ell-\xi^2 \lfloor\frac{\ell^\frac{1}{2}}{\xi}\rfloor^2}}
\\
&+ 2\int\limits_{\Big\lfloor\frac{\ell^{\frac{1}{2}}}{\xi}\Big\rfloor}^{\Big\lfloor\frac{\tilde{\ell}^{\frac{1}{2}}}{\xi}\Big\rfloor}
\sqrt{\tilde{\ell}-\xi^2t^2}\di t
\leqslant\frac{\pi}{2\xi}(\tilde{\ell}-\ell)+\sqrt{\tilde{\ell}-\ell}.
\end{align*}
The proof is complete.
\end{proof}

\subsection{Coefficient $a_p$.}
In this subsection we calculate the coefficients $a_p$ and estimate them.

As in the previous subsection, by (\ref{3.7}) and the parity of $N_0$ we have
\begin{align*}
a_p(\ell)
=& 2  \int\limits_{0}^{\tauo} \sum\limits_{\substack{
n\in\mathds{Z}_+,\, m\in\mathds{N}
\\
(n+\tau)^2+\xi^2 m^2\leqslant \ell
}}\cos 2\pi p \tau \di\tau
+ 2  \int\limits_{0}^{\tauo} \sum\limits_{\substack{
n\in\mathds{Z}_+,\, m\in\mathds{N}
\\
(n+1-\tau)^2 +\xi^2 m^2\leqslant \ell}
} \cos 2\pi p \tau \di\tau
\\
=& 2 \int\limits_{0}^{\tauo} \sum\limits_{\substack{
n\in\mathds{Z}_+,\, m\in\mathds{N}
\\
n^2+\xi^2 m^2\leqslant \ell<\left(n+\tauo\right)^2+\xi^2 m^2
\\
(n+\tau)^2+\xi^2 m^2\leqslant \ell
}}\cos 2\pi p \tau \di\tau
\\
&+ 2  \int\limits_{0}^{\tauo} \sum\limits_{\substack{
n\in\mathds{Z}_+,\, m\in\mathds{N}
\\
\left(n+\tauo\right)^2 +\xi^2 m^2\leqslant \ell<(n+1)^2+\xi^2 m^2
\\
(n+1-\tau)^2+\xi^2 m^2\leqslant \ell}
} \cos 2\pi p \tau \di\tau
\\
=&2 \sum\limits_{\substack{m=1,\ldots,\left\lfloor\frac{\ell^{\frac{1}{2}}}{\xi}\right\rfloor
\\
 n=\left\lfloor\sqrt{\ell-\xi^2 m^2}\right\rfloor
\\
0\leqslant  \left\lceil\sqrt{\ell-\xi^2 m^2}\right\rceil<\tauo}} \int\limits_{0}^{ \left\lceil\sqrt{\ell-\xi^2 m^2}\right\rceil } \cos 2\pi p \tau\di\tau
\\
&+ 2  \sum\limits_{\substack{m=1,\ldots,\left\lfloor\frac{\ell^{\frac{1}{2}}}{\xi}\right\rfloor
\\
 n=\left\lfloor\sqrt{\ell-\xi^2 m^2}\right\rfloor
\\
\tauo\leqslant  \left\lceil\sqrt{\ell-\xi^2 m^2}\right\rceil<1}} \int\limits_{1-\left\lceil\sqrt{\ell-\xi^2 m^2}\right\rceil}^{\tauo } \cos 2\pi p \tau\di\tau,
\end{align*}
and thus,
\begin{equation*}
a_p(\ell)=\frac{1}{\pi p} \sum\limits_{m=1}^{\left\lfloor\frac{\ell^{\frac{1}{2}}}{\xi}\right\rfloor}
\sin 2\pi p \sqrt{\ell-\xi^2 m^2}.
\end{equation*}

Our next step is to transform the above identity to an integral form. We again employ the Euler-Maclaurin formula for the function $t\mapsto \sin 2\pi p \sqrt{\ell-\xi^2t^2}$, see \cite[Ch. 1, Sect. 1.1, Thm. 1.3]{Kr}:
\begin{align*}
a_p(\ell):=&\frac{1}{\pi p} \int\limits_{0}^{\frac{\ell^{\frac{1}{2}}}{\xi}} \sin 2\pi p\sqrt{\ell-\xi^2 t^2}\di t - \frac{\sin (2\pi p \ell^{\frac{1}{2}})}{\sqrt{2}\pi p}
 \\
 &- 2 \int\limits_{0}^{\frac{\ell^{\frac{1}{2}}}{\xi}} \frac{\xi^2 t\,\phi(t)}{\sqrt{\ell-\xi^2 t^2}}\cos 2\pi p \sqrt{\ell-\xi^2 t^2}\di t,\qquad \phi(t):=\lceil t\rceil-\tfrac{1}{2}.
\end{align*}
In both integrals we make the change $t\mapsto \frac{\ell^{\frac{1}{2}}}{\xi}\sin t$:
\begin{align}\label{4.2}
&a_p(\ell):=S^{(1)}_p(\ell) + S^{(2)}_p(\ell) -  \frac{\sin(2\pi p \ell^{\frac{1}{2}})}{2\pi p},
\\
& S^{(1)}_p(\ell):=\frac{ \ell^{\frac{1}{2}}} {\pi p
\xi} \int\limits_{0}^{\frac{\pi}{2}} \sin (2\pi p \ell^{\frac{1}{2}}\cos t)\cos t \di t,\nonumber
\\
& S^{(2)}_p(\ell):=-2 \ell^{\frac{1}{2}} \int\limits_{0}^{\frac{\pi}{2}}   \phi\left(\frac{\ell^{\frac{1}{2}}}{\xi}\sin t\right) \cos(2\pi p \ell^{\frac{1}{2}}\cos t)\sin t \di t.\nonumber
\end{align}
Here the first integral is the well-known representation for the Bessel function:
\begin{equation}\label{4.3}
S^{(1)}_p(\ell)=\frac{\ell^{\frac{1}{2}}}{2p\xi} J_1(2\pi p\ell^{\frac{1}{2}}).
\end{equation}
The results of \cite[Ch. V\!I\!I, Sect. 7.3]{Wa} imply the estimate
\begin{equation*}
\left|J_1(t)+\sqrt{\frac{2}{\pi t}} \cos\left(t+\frac{\pi}{4}\right)\right| \leqslant \frac{\sqrt{2}}{8\sqrt{\pi}} \left(\frac{3}{t^{\frac{3}{2}}} \left|\cos\left(t-\frac{\pi}{4}\right)\right| +\frac{5}{16t^{\frac{5}{2}}} \left|\cos\left(t+\frac{\pi}{4}\right)\right|\right)
\end{equation*}
as $t>0$. By (\ref{4.3}) and the Cauchy-Schwarz inequality this leads us to the estimate
\begin{equation}\label{4.7}
\begin{aligned}
\left| S^{(1)}_p(\ell) + \frac{\ell^{\frac{1}{4}}}{2\pi p^{\frac{3}{2}} \xi} \cos\left(2\pi p \ell^{\frac{1}{2}} + \frac{\pi}{4}\right)
\right| \leqslant & \frac{1}{32\pi^2 \xi} \bigg(
\frac{3}{p^{\frac{5}{2}}\ell^{\frac{1}{4}}} \left|\cos \left(2\pi p\ell^{\frac{1}{2}}-\frac{\pi}{4}\right)\right|
 \\
 &+ \frac{5}{32\pi p^{\frac{7}{2}}\ell^{\frac{3}{4}}} \left|\cos\left(2\pi p \ell^{\frac{1}{2}}+\frac{\pi}{4}\right)\right|
\bigg)
\\
\leqslant & \frac{1}{32\pi^2 p^\frac{5}{2} \xi \ell^\frac{1}{4}} \sqrt{9  +  \frac{25}{1024\pi^2 p^2 \ell}}
\end{aligned}
\end{equation}
for all $\ell>0$.

\begin{lemma}\label{lm3.0}
The identity holds:
\begin{equation*}
S^{(2)}_p(\ell)=\frac{2 \ell^{\frac{1}{2}}}{\pi} \sum\limits_{k=1}^{\infty} \frac{1}{k}\int\limits_{0}^{\frac{\pi}{2}} \sin \left(\frac{2\pi k \ell^{\frac{1}{2}}}{\xi}\sin t\right) \cos(2\pi p \ell^{\frac{1}{2}}\cos t) \sin t \di t.
\end{equation*}
\end{lemma}

\begin{proof}
We represent the function $\phi$ by its Fourier series
\begin{equation*}
\phi(z)=-\frac{1}{\pi} \sum\limits_{k=1}^{\infty} \frac{\sin 2\pi k z}{k}
\end{equation*}
and we know that
\begin{equation}\label{3.30}
\lim\limits_{N\to+\infty} \int\limits_{0}^{R} \left|\phi(z)+ \sum\limits_{k=1}^{N} \frac{\sin 2\pi k z}{\pi k}\right|^2\di z=0
\end{equation}
for each fixed $R>0$. It was also stated in \cite[Ch. 1, Sect. 1.1]{Kr} that the estimate
\begin{equation}\label{3.31}
\left|\frac{1}{\pi} \sum\limits_{k=1}^{N} \frac{\sin 2\pi k z}{k}\right|\leqslant C,\qquad z\in[0,R],
\end{equation}
is true,  where $C$ is some constant independent of $N$ and $z$.

By the H\"older inequality we obtain
\begin{align*}
\frac{1}{2\ell^{\frac{1}{2}}}&\left|S_p^{(2)}(\ell)-\sum\limits_{k=1}^{N}\frac{1}{\pi k}
 \int\limits_{0}^{\frac{\pi}{2}} \sin \left(\frac{2\pi k \ell^{\frac{1}{2}}}{\xi}\sin t\right) \cos(2\pi p \ell^{\frac{1}{2}}\cos t) \sin t \di t
\right|^2
\\
&=\left| \int\limits_{0}^{\frac{\pi}{2}} \left(\phi\left(\frac{\ell^\frac{1}{2}}{\xi}\sin t\right)+
\sum\limits_{k=1}^{N} \frac{\sin \left(\frac{2\pi k \ell^{\frac{1}{2}}}{\xi}\sin t\right)}{\pi k}
\right) \cos(2\pi p \ell^{\frac{1}{2}}\cos t) \sin t \di t
\right|^2
\\
&\leqslant C  \int\limits_{0}^{\frac{\pi}{2}} \left|\phi\left(\frac{\ell^\frac{1}{2}}{\xi}\sin t\right)+
\sum\limits_{k=1}^{N} \frac{\sin \left(\frac{2\pi k \ell^{\frac{1}{2}}}{\xi}\sin t\right)}{\pi k}
\right|^2\di t
\\
&= C  \int\limits_{0}^{\frac{\ell^\frac{1}{2}}{\xi}} \left|\phi(z)+
 \sum\limits_{k=1}^{N} \frac{\sin(2\pi k z)}{\pi k }
\right|^2\frac{\di z}{\sqrt{\ell-\xi^2 z^2}}
\\
&\leqslant C \left(\int\limits_{0}^{\frac{\ell^\frac{1}{2}}{\xi}} \left|\phi(z)+
 \sum\limits_{k=1}^{N} \frac{\sin(2\pi k z)}{\pi k }
\right|^6\di z\right)^\frac{1}{3}
\left(\int\limits_{0}^{\frac{\ell^\frac{1}{2}}{\xi}}  \frac{\di z}{
(\ell-\xi^2 z^2)^\frac{3}{4}}
\right)^\frac{2}{3},
\end{align*}
where the symbol $C$ stands for some inessential constants independent of $N$. The function $\phi(z)$ is uniformly bounded and employing estimate (\ref{3.31}), we continue estimating as follows:
\begin{align*}
\frac{1}{2\ell^{\frac{1}{2}}}&\left|S_p^{(2)}(\ell)-\sum\limits_{k=1}^{N}\frac{1}{\pi k}
 \int\limits_{0}^{\frac{\pi}{2}} \sin \left(\frac{2\pi k \ell^{\frac{1}{2}}}{\xi}\sin t\right) \cos(2\pi p \ell^{\frac{1}{2}}\cos t) \sin t \di t
\right|^2
\\
&\leqslant C \left(\int\limits_{0}^{\frac{\ell^\frac{1}{2}}{\xi}} \left|\phi(z)+
 \sum\limits_{k=1}^{N} \frac{\sin(2\pi k z)}{\pi k }
\right|^2\di z\right)^\frac{1}{3}.
\end{align*}
By (\ref{3.30}), the right hand side of the obtained inequality tends to zero as $N\to+\infty$ and this proves the lemma.
\end{proof}


By the formula
\begin{align*}
\sin \left(\frac{2\pi k \ell^{\frac{1}{2}}}{\xi}\sin t\right) \cos(2\pi p \ell^{\frac{1}{2}}\cos t) = &\frac{1}{2} \sin\left( 2\pi   \ell^{\frac{1}{2}}\left(\frac{k}{\xi}\sin t+p \cos t\right)\right)
\\
&+\frac{1}{2}\sin\left( 2\pi \ell^{\frac{1}{2}} \left(\frac{k}{\xi}\sin t-p\cos t\right)\right)
\end{align*}
we get:
\begin{gather}\label{4.4}
S^{(2)}_p(\ell)=\frac{\ell^\frac{1}{2}}{\pi} \sum\limits_{k=1}^{\infty} \frac{1}{k} \big(S_{p,+}^{(2,k)}(\ell) + S_{p,-}^{(2,k)}(\ell)\big),
\\
S_{p,\pm}^{(2,k)}(\ell):= \int\limits_{0}^{\frac{\pi}{2}} \sin\left( 2\pi \ell^{\frac{1}{2}} \left(\frac{k}{\xi}\sin t\pm\ p\cos t\right)\right) \sin t \di t.\nonumber
\end{gather}

We denote
\begin{equation*}
  \a_{p,k}=\a_{p,k}(T):=\arctan\frac{p \xi}{k},\quad \eta_{p,k}=\eta_{p,k}(\ell,T):= 2\pi  \ell^{\frac{1}{2}} \sqrt{\frac{k^2}{\xi^2}+p^2}.
\end{equation*}
Then the formulae for $S_{p,\pm}^{(2,k)}$ can be rewritten as
\begin{align*}
S_{p,\pm}^{(2,k)}(\ell)=\int\limits_{0}^{\frac{\pi}{2}} \sin(\eta_{p,k}\sin(t\pm \a_{p,k}))\sin t \di t= \int\limits_{\pm \a_{p,k}}^{\frac{\pi}{2}\pm \a_{p,k}} \sin(\eta_{p,k}\sin t)\sin (t\mp \a_{p,k}) \di t.
\end{align*}
Hence, thanks to the parity properties of $\sin t$ and $\cos t$,
\begin{align*}
S_{p,+}^{(2,k)}+ S_{p,-}^{(2,k)}=& \cos\a_{p,k} \left( \int\limits_{-\a_{p,k}}^{\frac{\pi}{2}-\a_{p,k}} \sin(\eta_{p,k}\sin t)\sin t  \di t + \int\limits_{\a_{p,k}}^{\frac{\pi}{2}+\a_{p,k}} \sin(\eta_{p,k}\sin t)\sin t  \di t \right)
\\
&+ \sin\a_{p,k} \left(
\int\limits_{-\a_{p,k}}^{\a_{p,k}} \sin(\eta_{p,k}\sin t)\cos t  \di t - \int\limits_{\frac{\pi}{2}-\a_{p,k}}^{\frac{\pi}{2}+\a_{p,k}} \sin(\eta_{p,k}\sin t)\cos t  \di t
\right)
\\
=&2\cos\a_{p,k} \int\limits_{0}^{\frac{\pi}{2}} \sin(\eta_{p,k}\sin t)\sin t  \di t.
\end{align*}
In the last integral we make the change of the variable $t\mapsto \sin\left(\frac{\pi}{4}-\frac{t}{2}\right)$ and we get:
\begin{equation}\label{3.43}
S_{p,+}^{(2,k)} + S_{p,-}^{(2,k)}=4\cos\a_{p,k}\int\limits_{0}^{\frac{1}{\sqrt{2}}}  \sin(\eta_{p,k}(1-2t^2))\frac{1-2t^2}{\sqrt{1-t^2}}\di t.
\end{equation}
We denote
\begin{equation*}
h(s):= \frac{2}{\sqrt{1-s}} - \frac{1}{1-s+\sqrt{1-s}}
\end{equation*}
and we see that
\begin{equation*}
\frac{1-2t^2}{\sqrt{1-t^2}}=1-t^2h(t^2).
\end{equation*}
We substitute this identity into the integral in (\ref{3.43}):
\begin{gather}\label{3.44}
S_{p,+}^{(2,k)} + S_{p,-}^{(2,k)}=\cos\a_{p,k} \big(
S_p^{(4,k)} - S_p^{(3,k)}\big),
\\
S_p^{(3,k)}:=4\int\limits_{0}^{ \frac{1}{\sqrt{2}}} \sin(\eta_{p,k}(1-2t^2)) h(t^2) t^2 \di t,\quad
S_p^{(4,k)}:=4\int\limits_{0}^{\frac{1}{\sqrt{2}}} \sin(\eta_{p,k}(1-2t^2)) \di t. \nonumber
\end{gather}

In the integral $S_p^{(3,k)}$ we integrate by parts as follows:
\begin{equation}\label{4.11}
\begin{aligned}
S_p^{(3,k)}=&\frac{1}{\eta_{p,k}} \int\limits_{0}^{\frac{1}{\sqrt{2}}} h(t^2)t\di \cos (\eta_{p,k}(1-2t^2))
\\
=& \frac{\sqrt{2}}{\eta_{p,k}}
 -\frac{1}{\eta_{p,k}} \int\limits_{0}^{\frac{1}{\sqrt{2}}} \big(h(t^2)+2t^2 h'(t^2)\big) \cos(\eta_{p,k}(1-2t^2))\di t.
\end{aligned}
\end{equation}
The function
\begin{equation*}
h(s)+2s h'(s)=\frac{1}{(1+\sqrt{1-s})^2} + \frac{2}{(1-s)(1+\sqrt{1-s})} + \frac{1}{(1-s)^\frac{3}{2}(1+\sqrt{1-s})^2}
\end{equation*}
grows monotonically as $s\in[0,\tfrac{1}{2}]$ and hence,
\begin{equation*}
0<\frac{3}{2}=h(0)\leqslant h(t^2)+2t^2 h'(t^2).
\end{equation*}
Employing this estimate, we get
\begin{equation*}
\left|\int\limits_{0}^{\frac{1}{\sqrt{2}}} \big(h(t^2)+2t^2 h'(t^2)\big) \cos\eta_{p,k}(1-2t^2)\di t\right| \leqslant \int\limits_{0}^{\frac{1}{\sqrt{2}}} \big(h(t^2)+2t^2 h'(t^2)\big)\di t
 =\sqrt{2},
\end{equation*}
and by (\ref{4.11}) we obtain
\begin{equation*}
|S_p^{(3,k)}|\leqslant \frac{2\sqrt{2}}{\eta_{p,k}}.
\end{equation*}
In view of the definition of $\a_{p,k}$ we have
\begin{equation}\label{3.24}
\frac{\cos\a_{p,k}}{k}=\frac{1}{(k^2+p^2\xi^2)^\frac{1}{2}}.
\end{equation}
Hence,
\begin{equation}\label{3.45}
\left|\sum\limits_{k=1}^{+\infty} \frac{\cos\a_{p,k}}{k} S_p^{(3,k)}\right|\leqslant \frac{\sqrt{2}\xi}{\pi \ell^{\frac{1}{2}}}  \sum\limits_{k=1}^{\infty} \frac{1}{k^2+p^2\xi^2}= \frac{1}{\sqrt{2}\ell^{\frac{1}{2}}} \left(\frac{\coth(\pi p\xi)}{p\xi}-\frac{1}{\pi p^2\xi^2}\right).
\end{equation}

We proceed to estimating $S_p^{(4,k)}$.

\begin{lemma}\label{lm3.5}
The estimate
\begin{align*}
\left|\sum\limits_{k=1}^{\infty}\frac{\cos\a_{p,k}}{k}S_p^{(4,k)} -
\frac{\xi^\frac{1}{2}}{\ell^\frac{1}{4}} \sum\limits_{k=1}^{\infty} \frac{\sin\left(\eta_{p,k}-\frac{\pi}{4}\right)}{(k^2+p^2\xi^2)^\frac{3}{4}}
\right|
\leqslant \frac{1}{\sqrt{2}\ell^{\frac{1}{2}}} \left(\frac{\coth(\pi p\xi)}{p\xi}-\frac{1}{\pi p^2\xi^2}\right)
\end{align*}
holds true.
\end{lemma}

\begin{proof}
 We make the change of the variable $t \mapsto \sqrt{2}\eta_{p,k}^\frac{1}{2}t$ in the integral $S_p^{(4,k)}$:
\begin{equation}\label{3.23}
\begin{aligned}
S_p^{(4,k)}=&\frac{2\sqrt{2}}{ \eta_{p,k}^{\frac{1}{2}}} \int\limits_{0}^{\eta_{p,k}^\frac{1}{2}} \sin(\eta_{p,k}-t^2)\di t
\\
=&\frac{\sqrt{2\pi}}{\eta_{p,k}^\frac{1}{2}}
 \sin\left(\eta_{p,k}-\frac{\pi}{4}\right)   - \frac{2\sqrt{2}}{\eta_{p,k}^\frac{1}{2}} \int\limits_{ \eta_{p,k}^\frac{1}{2}}^{+\infty} \sin(\eta_{p,k}-t^2)\di t.
\end{aligned}
\end{equation}
For the latter integral we have:
\begin{align*}
-&\int\limits_{\eta_{p,k}^{\frac{1}{2}}}^{+\infty}\sin (\eta_{p,k}-t^2)\di t = -\int\limits_{\eta_{p,k}^{\frac{1}{2}}}^{+\infty} \frac{d\cos(\eta_{p,k}-t^2)}{2t}= \frac{1}{2\eta_{p,k}^{\frac{1}{2}}} - \int\limits_{\eta_{p,k}^{\frac{1}{2}}}^{+\infty} \frac{\cos(\eta_{p,k}-t^2)}{2t^2}\di t,
\\
&\left|\int\limits_{\eta_{p,k}^{\frac{1}{2}}}^{+\infty}\frac{\cos(\eta_{p,k} -t^2)}{2t^2}\di t\right|\leqslant \frac{1}{2\eta_{p,k}^{\frac{1}{2}}},
 \end{align*}
and therefore,
\begin{equation*}
0\leqslant -\int\limits_{\eta_{p,k}^{\frac{1}{2}}}^{+\infty}\sin (\eta_{p,k}-t^2)\di t\leqslant \frac{1}{\eta_{p,k}^{\frac{1}{2}}}.
\end{equation*}
Hence, by (\ref{3.24}) and  (\ref{3.23}) we get
\begin{align*}
\left| \sum\limits_{k=1}^{\infty}\frac{\cos\a_{p,k}}{k}S_p^{(4,k)} -
\frac{\xi^\frac{1}{2}}{\ell^\frac{1}{4}} \sum\limits_{k=1}^{\infty} \frac{\sin\left(\eta_{p,k}-\frac{\pi}{4}\right)}{(k^2+p^2\xi^2)^\frac{3}{4}}
\right| \leqslant &\frac{\sqrt{2}\xi}{\pi \ell^\frac{1}{2}}\sum\limits_{k=1}^{\infty} \frac{1}{k^2+p^2\xi^2}
 \\
 =& \frac{1}{\sqrt{2}\ell^{\frac{1}{2}}} \left(\frac{\coth(\pi p\xi)}{p\xi}-\frac{1}{\pi p^2\xi^2}\right).
\end{align*}
The proof is complete.
\end{proof}

Identities (\ref{4.4}), (\ref{3.43}), (\ref{3.44}), estimate (\ref{3.45}) and  Lemma~\ref{lm3.5} yield:
\begin{equation}\label{3.16a}
\left|S_p^{(2)}(\ell)-
\frac{\xi^\frac{1}{2}\ell^\frac{1}{4}}{\pi} \sum\limits_{k=1}^{\infty} \frac{\sin\left(\eta_{p,k}-\frac{\pi}{4}\right)}{(k^2+p^2\xi^2)^\frac{3}{4}}
\right| \leqslant \frac{\sqrt{2}}{ \ell^\frac{1}{2}} \left(\frac{\coth(\pi p\xi)}{p\xi}-\frac{1}{\pi p^2\xi^2}\right).
\end{equation}

The derivative
\begin{align*}
\left(\frac{\coth z}{z}-\frac{1}{z^2}\right)'=\frac{1}{z^2\sinh z} \left(1-\frac{z}{\sinh z} + \frac{4\cosh \frac{z}{2}}{z} \left(\sinh\frac{z}{2}-\frac{z}{2}\cosh\frac{z}{2}\right)
\right)
\end{align*}
is non-positive since
\begin{equation*}
1-\frac{z}{\sinh z}\leqslant 0,\qquad \sinh z - z\cosh z\leqslant 0.
\end{equation*}
Two latter inequalities can be easily checked by calculating the values at zero and the derivatives of the functions in their left hand sides. Then
\begin{equation*}
0\leqslant\frac{\sqrt{2}}{\pi} \left(\frac{\coth\pi z}{z}-\frac{1}{\pi z^2}\right)\leqslant \lim\limits_{z\to 0} \frac{\sqrt{2}}{\pi} \left(\frac{\coth\pi z}{z}-\frac{1}{\pi z^2}\right)=\frac{\sqrt{2}}{3}.
\end{equation*}
Hence, by (\ref{4.2}), (\ref{4.7}), (\ref{3.16a}) and definition (\ref{2.4}) of the function $\vp_p$  we infer that the coefficient $a_p$ can be represented as
\begin{align}\label{3.28}
&\;a_p(\ell)=\ell^\frac{1}{4}\vp_p(\xi,\ell) +S_p^{(5)}(\ell),
\\
&
\begin{aligned}
|S_p^{(5)}(\ell)|\leqslant & \frac{\sqrt{2}}{3} + \frac{1}{2\pi} +  \frac{1}{32\pi^2\xi \ell^\frac{1}{4}p^\frac{5}{2}}\sqrt{9+\frac{25}{1024\pi^2p^2 \ell}}.
\end{aligned}\label{3.19}
\end{align}

\section{Finitely many gaps}

In this section we prove Theorems~\ref{th2.1},~\ref{th2.3}. We follow the lines of work \cite{DT} with certain minor modifications.

We begin with an auxiliary lemma.

\begin{lemma}\label{lm4.1}
The estimates
\begin{align*}
&\sup\limits_{\tau\in\left[-\tauo,\tauo\right]} N_0(\ell,\tau) \geqslant a_0(\ell) + \frac{1}{2}\sup\limits_{p\in\mathds{N}}\{|a_p(\ell)|\},
\\
&\inf\limits_{\tau\in\left[-\tauo,\tauo\right]} N_0(\ell,\tau) \leqslant a_0(\ell) - \frac{1}{2}\sup\limits_{p\in\mathds{N}}\{|a_p(\ell)|\},
\end{align*}
hold true.
\end{lemma}

\begin{proof}
We introduce the functions
\begin{align*}
&\tilde{N}^0(\tau,\ell):=N_0(\ell,\tau)-a_0(\ell),
\\
&\tilde{N}^0_+(\tau,\ell):=\max\{\tilde{N}^0(\tau,\ell),0\},
\\
&\tilde{N}^0_-(\tau,\ell):=\min\{\tilde{N}^0(\tau,\ell),0\}.
\end{align*}
These functions obey  the identities
\begin{align*}
\int\limits_{-\tauo}^{\tauo} \tilde{N}^0(\tau,\ell)\di\tau=&0,
\\
\int\limits_{-\tauo}^{\tauo} |\tilde{N}^0(\tau,\ell)|\di\tau = & \int\limits_{-\tauo}^{\tauo} \tilde{N}^0_+(\tau,\ell)\di\tau - \int\limits_{-\tauo}^{\tauo} \tilde{N}^0_-(\tau,\ell)\di\tau
\\
=& 2 \int\limits_{-\tauo}^{\tauo} \tilde{N}^0_+(\tau,\ell)\di\tau = -2 \int\limits_{-\tauo}^{\tauo} \tilde{N}^0_-(\tau,\ell)\di\tau
\end{align*}
and therefore,
\begin{align*}
&\inf\limits_{\tau\in\left[-\tauo,\tauo\right]} \tilde{N}^0_-(\tau,\ell)=\inf\limits_{\tau\in\left[-\tauo,\tauo\right]} \tilde{N}^0(\tau,\ell)\leqslant 0,
\\
&\hphantom{_{\tau\in[.}} 0\leqslant \sup\limits_{\tau\in\left[-\tauo,\tauo\right]} \tilde{N}^0(\tau,\ell)=\sup\limits_{\tau\in\left[-\tauo,\tauo\right]} \tilde{N}^0_+(\tau,\ell).
\end{align*}
Hence, by definition (\ref{3.8}) of $a_p(\ell)$
we obtain immediately
\begin{align*}
& \sup\limits_{\tau\in\left[-\tauo,\tauo\right]} \tilde{N}^0(\tau,\ell)\geqslant
\int\limits_{-\tauo}^{\tauo} \tilde{N}^0_+(\tau,\ell) \di\tau=\frac{1}{2}\int\limits_{-\tauo}^{\tauo} |\tilde{N}^0(\tau,\ell)|\di\tau\geqslant \frac{1}{2}|a_p(\ell)|,
\\
& \inf\limits_{\tau\in\left[-\tauo,\tauo\right]} \tilde{N}^0(\tau,\ell)\leqslant
\int\limits_{-\tauo}^{\tauo} \tilde{N}^0_-(\tau,\ell) \di\tau=-\frac{1}{2}\int\limits_{-\tauo}^{\tauo} |\tilde{N}^0(\tau,\ell)|\di\tau\leqslant -\frac{1}{2}|a_p(\ell)|.
\end{align*}
These inequalities and the definition of the function $\tilde{N}^0(\tau,\ell)$ imply the statement of the lemma.
\end{proof}

By identity (\ref{3.28}) and inequality (\ref{3.19}) we can estimate the supremum of $|a_p(\ell)|$ from below:
\begin{align}\label{4.14}
&
\begin{aligned}
\sup\limits_{p} \{|a_p(\ell)|\}  \geqslant &\ell^{\frac{1}{4}} \sup\limits_{p}\{|\vp_p(\xi,\ell)|\} - \sup\limits_{p} \{|S_p^{(5)}(\ell)|\} \\
\geqslant & \ell^{\frac{1}{4}} \sup\limits_{p}\{|\vp_p(\xi,\ell)|\}  - S^{(6)}(\ell),
\end{aligned}
\\
&
\begin{aligned}
S^{(6)}(\ell)=  \frac{\sqrt{2}}{3} + \frac{1}{2\pi} +  \frac{1}{32\pi^2\xi \ell^\frac{1}{4}}\sqrt{9+\frac{25}{1024\pi^2 \ell}}.
\end{aligned}\label{4.16a}
\end{align}

It is easy to confirm that
\begin{gather*}
S^{(6)}(\ell)\leqslant \frac{c_0}{2}\ell^{\frac{1}{4}-\g} \quad\text{as}\quad \ell\geqslant \ell_{2},
\\
\ell_{2}:=\max\left\{\ell_0,\left(\frac{4\sqrt{2}\pi+6}{3\pi c_0}\right)^\frac{4}{1-4\g},\left(\frac{1}{8\pi^2\xi c_0}\sqrt{9+\frac{25}{1024\pi^2}}\right)^\frac{2}{1-2\g}\right\},
\end{gather*}
where $\g$ comes from condition (\ref{2.3}). Hence, by condition (\ref{2.3}) and Lemma~\ref{lm4.1}, as $\ell\geqslant \ell_2$,
the estimates hold:
\begin{equation}\label{4.16}
\begin{aligned}
&\sup\limits_{\tau\in\left[-\tauo,\tauo\right]} N_0(\ell,\tau) \geqslant a_0(\ell) + \frac{c_0}{2} \ell^{\frac{1}{4}-\g},
\\
&\inf\limits_{\tau\in\left[-\tauo,\tauo\right]} N_0(\ell,\tau) \leqslant a_0(\ell) -
\frac{c_0}{2}\ell^{\frac{1}{4}-\g}.
\end{aligned}
\end{equation}

Let $[\eta_k^0,\tht_k^0]$, $k\geqslant 1$, be the $k$th band of the operator $\Op$ in the case $\cL=0$, that is,
\begin{equation*}
\min\limits_{\tau\in\left[-\tauo,\tauo\right]} E_k^0(\tau)=\eta_k^0,\qquad \max\limits_{\tau\in\left[-\tauo,\tauo\right]} E_k^0(\tau)=\tht_k^0.
\end{equation*}
By the definition of the counting function $N_0(\ell,\tau)$, for a fixed $\ell$, the number of the band functions $E_k^0(\tau)$ whose minima do not exceed $\frac{\pi^2 \ell}{T^2}$ is equal to $\sup\limits_{\tau\in\left[-\tauo,\tauo\right]} N_0(\ell,\tau)$, while $\inf\limits_{\tau\in\left[-\tauo,\tauo\right]} N_0(\ell,\tau)$ is the number of the band functions $E_k^0(\tau)$ whose maxima do not exceed $\frac{\pi^2 \ell}{T^2}$. Hence, for each $k\geqslant 1$,
\begin{equation*}
\sup\limits_{\tau\in\left[-\tauo,\tauo\right]} N_0\left(\frac{T^2 \eta_k^0}{\pi^2},\tau\right)=k,\qquad \inf\limits_{\tau\in\left[-\tauo,\tauo\right]} N_0\left(\frac{T^2 \tht_k^0}{\pi^2},\tau\right)=k.
\end{equation*}
Assuming now
\begin{equation}\label{5.13}
\eta_k^0\geqslant \frac{\pi^2}{T^2} \ell_{2},
\end{equation}
by (\ref{4.16}) we obtain
\begin{equation*}
k+1 \geqslant a_0\left(\frac{T^2}{\pi^2}\eta_{k+1}^0\right) + \frac{c_0}{2}\left(\frac{T^2}{\pi^2}\eta_{k+1}^0\right)^{\frac{1}{4}-\g},
\quad
k\leqslant a_0\left(\frac{T^2}{\pi^2}\tht_k^0\right) - \frac{c_0}{2}\left(\frac{T^2}{\pi^2}\tht_k^0\right)^{\frac{1}{4}-\g}.
\end{equation*}
The operator $\Op$ as $\cL=0$ is the Dirichlet Laplacian and its spectrum has no internal spectral gaps. Therefore, $\eta_{k+1}^0\leqslant \tht_k^0$ and by Lemma~\ref{lm3.2}
and the inequality
\begin{equation*}
\a^2+2\a\b\leqslant \frac{4}{3}\a^2+3\b^2,\qquad \a,\b\geqslant0,
\end{equation*}
this implies:
\begin{equation}\label{4.14a}
\begin{aligned}
\frac{T}{2\pi\xi}\left(\frac{4}{3}(\tht_k^0-\eta_{k+1}^0) + 3\xi^2
\right)\geqslant &
\frac{T}{2\pi\xi}(\tht_k^0-\eta_{k+1}^0)+\frac{T}{\pi}\sqrt{
\tht_k^0-\eta_{k+1}^0
}
\\
\geqslant & a_0\left(\frac{T^2}{\pi^2}\tht_k^0\right) - a_0\left(\frac{T^2}{\pi^2}\eta_{k+1}^0\right)
\\
\geqslant & c_0\left(\frac{T}{\pi}\right)^{\frac{1}{2}-2\g}\big(
(\tht_k^0)^{\frac{1}{4}-\g} + (\eta_{k+1}^0)^{\frac{1}{4}-\g}\big)-1
\\
\geqslant & 2
c_0\left(\frac{T^2}{\pi^2}\eta_k^0\right)^{\frac{1}{4}-\g}
-1.
\end{aligned}
\end{equation}
Hence,
\begin{equation}\label{5.12}
\tht_k^0-\eta_{k+1}^0\geqslant \frac{3
\pi \xi
c_0}{T}\left(\frac{T^2}{\pi^2}\eta_k^0\right)^{\frac{1}{4}-\g}
-\frac{3\pi\xi}{2T}-\frac{9\xi^2}{4}.
\end{equation}
Since $\eta_k^0\to+\infty$ as $k\to+\infty$, the above estimate means that the length of the overlapping of the bands in the spectrum of $\Op$ as $\cL=0$ grows as $k\to+\infty$.

Let $[\eta_k,\tht_k]$, $k\geqslant 1$, be the spectral bands of the operator $\Op$ for a given operator $\cL$. In view of definition (\ref{2.5}) of $\om_\pm$  and the minimax principle for each $k$ we have
\begin{equation}\label{5.31}
\eta_k^0+\om_- \leqslant \eta_k \leqslant \eta_k^0 + \om_+,\qquad
\tht_k^0+\om_- \leqslant \tht_k \leqslant \tht_k^0 + \om_+.
\end{equation}
Hence, by (\ref{5.12}), the bands $[\eta_k,\tht_k]$ overlap for sufficiently large $k$, namely, as
\begin{equation*}
\frac{3\pi \xi c_0}{T} \left(\frac{T^2}{\pi^2}\eta_k^0\right)^{\frac{1}{4}-\g}
-\frac{3\pi\xi}{2T}-\frac{9}{4}\xi^2\geqslant \om_\cL,
\end{equation*}
or, equivalently, as
\begin{equation*}
\left(\frac{T^2}{\pi^2}\eta_k^0\right)^{\frac{1}{4}-\g}
\geqslant \frac{T}{4c_0\pi\xi}\om_\cL+\frac{1}{2c_0}+\frac{3\xi T}{4\pi c_0}.
\end{equation*}
In addition, condition  (\ref{5.13}) should be satisfied.  Both these conditions are true if
\begin{equation*}
\eta_k^0\geqslant \ell_{1}-\om_-,
\end{equation*}
where $\ell_{1}$ was introduced in (\ref{2.11}).
And   by (\ref{5.31})  we conclude that the operator $\Op$ surely has no spectral gaps in $[\ell_1,+\infty)$. This completes the proof of Theorem~\ref{th2.1}.

We proceed to proving Theorem~\ref{th2.3}. We have:
\begin{align*}
&\left|\sum\limits_{k\in\mathds{Z}\setminus\{0\} }  \frac{\sin\left( 2\pi  \ell^{\frac{1}{2}} \sqrt{\frac{k^2}{\xi^2}+p^2} -\frac{\pi}{4}\right)}{\left(\frac{k^2}{\xi^2}+p^2\right)^\frac{3}{4}} \right| \leqslant 2 \sum\limits_{k=1}^{\infty}\frac{\xi^\frac{3}{2}}{k^\frac{3}{2}}= 2\z\left(\frac{3}{2}\right)\xi^\frac{3}{2}.
\end{align*}
Therefore,
\begin{equation*}
|\vp_p(\ell)|\geqslant \frac{\left|\sin\left( 2\pi  \ell^{\frac{1}{2}} p -\frac{\pi}{4}\right)\right|}{\pi\xi p^\frac{3}{2}} - 2\z\left(\frac{3}{2}\right)\xi^\frac{3}{2}
\end{equation*}
and
\begin{equation}\label{5.19}
\sup\limits_{p} |\vp_p(\ell)|\geqslant \frac{1}{\pi\xi} \left(\max\left\{
\left|\sin\left( 2\pi  \ell^{\frac{1}{2}}  -\frac{\pi}{4}\right)\right|,\,3^{-\frac{3}{2}}\left|\sin\left( 6\pi  \ell^{\frac{1}{2}}  -\frac{\pi}{4}\right)\right|
\right\}-2\z\left(\frac{3}{2}\right)\xi^\frac{3}{2}\right).
\end{equation}

Denote $z=2\pi \ell^\frac{1}{2}-\tfrac{\pi}{4}$, then
\begin{equation}\label{5.20}
\max\left\{ \left|\sin\left( 2\pi  \ell^{\frac{1}{2}}  -\frac{\pi}{4}\right)\right|,\,3^{-\frac{3}{2}}\left|\sin\left( 6\pi  \ell^{\frac{1}{2}}  -\frac{\pi}{4}\right)\right|
\right\}=\max\left\{
|\sin z|,\,3^{-\frac{3}{2}}|\cos 3z|
\right\}.
\end{equation}

The function $|\sin z|$ is $\pi$-periodic and
$3^{-\frac{3}{2}}|\cos 3z|$ is $\tfrac{\pi}{3}$-periodic and
$$
|\sin z|=|\sin(\pi-z)|,\qquad |\cos 3z|=|\cos 3(\pi-z)|,\qquad z\in[0,\pi].
$$
The function $|\sin z|$ increases from $0$ to $1$ and the function $|\cos 3z|$ decreases from $1$ to $0$ as $z\in[0,\tfrac{\pi}{6}]$. Then it is straightforward to confirm that
\begin{equation*}
\max\left\{
|\sin z|,\,3^{-\frac{3}{2}}|\cos 3z|
\right\} =
\begin{cases}
3^{-\frac{3}{2}}|\cos 3z|, \quad & z\in[0,z_0]\cup[\pi-z_0,\pi],
\\
|\sin z|,\quad & z_0\leqslant z\leqslant \pi-z_0,
\end{cases}
\end{equation*}
where $z_0\in(0,\tfrac{\pi}{6})$ is the root of the equation
\begin{equation}\label{5.22}
\sin z-3^{-\frac{3}{2}}\cos 3z=0.
\end{equation}
This implies immediately that
\begin{equation}\label{5.21}
\min\limits_{z\in[0,\pi]} \max\left\{
|\sin z|,\,3^{-\frac{3}{2}}|\cos 3z|
\right\} =\sin z_0=\frac{\tan z_0}{\sqrt{\tan^2 z_0+1}}.
\end{equation}
Equation (\ref{5.22}) is reduced to the third order equation for $\tan z$:
\begin{equation*}
3^\frac{3}{2}\tan^3 z + 3 \tan^2 z + 3^\frac{3}{2}\tan z - 1=0,
\end{equation*}
which can be solved explicitly:
$\tan z_0=c_2$.
Hence, by (\ref{5.21}),
\begin{equation*}
\min\limits_{z\in[0,\pi]}\max\left\{
|\sin z|,\,3^{-\frac{3}{2}}|\cos 3z|
\right\}= c_1:=\frac{c_2}{\sqrt{c_2^2+1}}
\end{equation*}
and it follows from (\ref{5.19}), (\ref{5.20}) that
\begin{equation*}
\sup\limits_{p} |\vp_p(\ell)|\geqslant \frac{c_1-2\z\left(\frac{3}{2}\right)\xi^\frac{3}{2}}{\pi\xi}.
\end{equation*}
This completes the proof.

\section{Absence of gaps}

In this section we prove Theorem~\ref{th2.2}.
Replacing the operator $\cL$ by $\tilde{\cL}:=\cL-\om_-$, we just shift the spectrum of the operator $\Op$ and therefore, it is sufficient to prove the theorem for the operator $\tilde{\cL}$. The advantage of using such operator instead of $\cL$ is that the constant $\om_-$ defined by (\ref{2.5}) is zero and we have
\begin{equation}\label{5.1}
0\leqslant (\tilde{\cL}u,u)_{\ell_2(\square)}\leqslant \om_\cL \|u\|_{\ell_2(\square)}^2
\end{equation}
for all $u\in \ell_2(\square)$. This is why from the very beginning we assume that for the operator $\cL$ we have $\om_-=0$ and inequality (\ref{5.1}) is satisfied.

The proof consists of two parts. In the first part we prove the absence of the gaps in the lower part of the spectrum, namely, below the point $\tfrac{\pi^2}{T^2}$. Here we employ the approach suggested in \cite[Sect. 5.2]{UMJ18}. In the second part of the proof we show the absence of the gaps in the higher part of the spectrum, that is, above the point $\tfrac{\pi^2}{T^2}$. This will be done by the approach employed in the previous section.

We begin with studying the lower part of the spectrum. Similar to (\ref{3.3}), we introduce the counting function for the operator $\Op$ with a given operator $\cL$:
\begin{equation*}
N(\ell,\tau)=\#\left\{E_k(\tau):\, E_k^0(\tau)\leqslant \frac{\pi^2 \ell}{T^2}\right\}.
\end{equation*}
By the minimax principle and (\ref{5.1}) we have the estimates
\begin{equation}\label{5.25}
E_k^0(\tau)\leqslant E_k(\tau)\leqslant E_k^0(\tau)+\om_\cL
\end{equation}
and therefore,
\begin{equation}\label{5.26}
N_0\left(\ell-\frac{T^2}{\pi^2}\om_\cL,\tau\right) \leqslant N(\ell,\tau) \leqslant N_0\left(\ell,\tau\right).
\end{equation}
The operator $\Op$ has no gaps in $\big[\inf\spec(\Op),\tfrac{\pi^2}{T^2}\big]$ if for all $\ell\in\big[\tfrac{T^2}{\pi^2}\inf\spec(\Op),1\big]$
the estimate holds:
\begin{equation*}
\sup\limits_{\tau\in\left[-\tauo,\tauo\right]} N(\ell,\tau)- \inf\limits_{\tau\in\left[-\tauo,\tauo\right]} N(\ell,\tau)\geqslant 1.
\end{equation*}
Since the function $N(\ell,\tau)$ is integer-valued, to ensure the above inequality, it is sufficient to find $\tau_{min},\tau_{max}\in[-\tauo,\tauo]$ such that
\begin{equation*}
N(\ell,\tau_{max})- N(\ell,\tau_{min})>0.
\end{equation*}
Hence, in view of (\ref{5.26}), it is sufficient to show that
\begin{equation}\label{5.27}
N_0\left(\ell-\frac{T^2}{\pi^2}\om_\cL,\tau_{max}\right) -
N_0\left(\ell,\tau_{min}\right)>0\quad \text{as}\quad \ell\in\big[\tfrac{T^2}{\pi^2}\inf\spec(\Op),1\big].
\end{equation}
Exactly this inequality will be checked in the first part of the proof.

It is straightforward to confirm that condition (\ref{2.9}) implies the estimate
\begin{equation}
0\leqslant \frac{T^2}{\pi^2}\om_\cL<\frac{1}{4}+\xi^2.\label{2.7}
\end{equation}
By (\ref{3.6}), the band function $E_1^0$ is given by the formula
\begin{equation*}
E_1^0(\tau)=\frac{\pi^2}{T^2}(\tau^2+\xi^2m^2)
\end{equation*}
and hence, due to (\ref{5.25}) and (\ref{2.7}), the first spectral band of $\Op$ is at least
$$
\left[\inf\spec(\Op),\, \frac{\pi^2}{T^2}\left(\frac{1}{4}+\xi^2\right)\right].
$$
This is why, in what follows we need to prove the absence of gaps only for \begin{equation*}
\l>\frac{\pi^2}{T^2}\left(\frac{1}{4}+\xi^2\right).
\end{equation*}
In terms of the parameter $\ell$ used in (\ref{5.27}), this means to study the case
\begin{equation*}
\ell>\frac{1}{4}+\xi^2.
\end{equation*}

Apart of (\ref{3.7}), the counting function $N_0(\ell,\tau)$ possesses one more representation:
\begin{equation}\label{5.29}
N_0(\ell,\tau)=\sum\limits_{n=-\lfloor\ell^\frac{1}{2}+\tau\rfloor}^{ \lfloor\ell^\frac{1}{2}-\tau\rfloor} \left\lfloor\frac{\sqrt{\ell-(n+\tau)^2}}{\xi}\right\rfloor.
\end{equation}

Consider the equation
\begin{equation}\label{5.30}
2\frac{\sqrt{\ell-\frac{T^2}{\pi^2}\om_\cL-\frac{1}{4}}}{\xi}-1 =\frac{\sqrt{\ell-\frac{T^2}{\pi^2}\om_\cL}}{\xi}.
\end{equation}
Its positive root is given by the formula
\begin{equation*}
\ell_*=\frac{(\sqrt{3+4\xi^2}+\xi)^2}{9} + \frac{T^2}{\pi^2}\om_\cL.
\end{equation*}
Conditions (\ref{2.7}), (\ref{2.8}) imply that
\begin{equation*}
\frac{1}{4}+\xi^2<\ell_*<\frac{2}{3}.
\end{equation*}
As $\ell\leqslant \ell_*$, we let $\tau_{max}:=0$, $\tau_{min}:=1-\ell^\frac{1}{2}$ in (\ref{5.27}) and by (\ref{5.29}) we obtain
\begin{equation*}
N_0(\ell,\tau_{min})=\left\lfloor\frac{\sqrt{2\ell^\frac{1}{2}-1}}{\xi}\right\rfloor,
\qquad N_0(\ell,\tau_{max})=\left\lfloor\frac{\sqrt{\ell-\frac{T^2}{\pi^2}\om_\cL}}{\xi}\right\rfloor.
\end{equation*}
Hence, thanks to condition (\ref{2.9}),
\begin{align*}
N_0\left(\ell-\frac{T^2}{\pi^2}\om_\cL,\tau_{max}\right) - N_0(\ell,\tau_{min})\geqslant & \frac{\sqrt{\ell-\frac{T^2}{\pi^2}\om_\cL}}{\xi}- \frac{\sqrt{2\ell^\frac{1}{2}-1}}{\xi}-1
\\
\geqslant & \frac{\sqrt{\ell_*-\frac{T^2}{\pi^2}\om_\cL}}{\xi}- \frac{\sqrt{2\ell_*^\frac{1}{2}-1}}{\xi}-1>0.
\end{align*}

As $\ell_*<\ell<1$, we choose $\tau_{max}=\tfrac{1}{2}$ and by (\ref{5.29}) we get
\begin{equation*}
N_0\left(\ell-\frac{T^2}{\pi^2}\om_\cL,\tau_{max}\right) = 2 \left(\frac{\sqrt{\ell-\frac{T^2}{\pi^2}\om_\cL-\frac{1}{4}}}{\xi}\right).
\end{equation*}
Hence, by (\ref{5.30}) and (\ref{2.9}),
\begin{align*}
N_0\left(\ell-\frac{T^2}{\pi^2}\om_\cL,\tau_{max}\right) - N_0(\ell,\tau_{min})\geqslant & \frac{2\sqrt{\ell-\frac{T^2}{\pi^2}\om_\cL-\frac{1}{4}}}{\xi}- \frac{\sqrt{2\ell^\frac{1}{2}-1}}{\xi}-2
\\
= & \frac{\sqrt{\ell_*-\frac{T^2}{\pi^2}\om_\cL}}{\xi}- \frac{\sqrt{2\ell_*^\frac{1}{2}-1}}{\xi}-1>0.
\end{align*}

In the remaining part of the proof we study the case $\ell\geqslant 1$ and here we shall employ the same approach as in the previous section.
Thanks to Theorem~\ref{th2.3} and condition (\ref{2.8}), estimate (\ref{2.3}) holds for $\ell\geqslant 1$. Then by (\ref{4.14}) and Lemma~\ref{lm4.1} we can improve estimates (\ref{4.16}):
\begin{align*}
&\sup\limits_{\tau\in\left[-\tauo,\tauo\right]} N_0(\ell,\tau) \geqslant a_0(\ell) + \frac{c_1-2\z\left(\frac{3}{2}\right)\xi^\frac{3}{2}}{\pi\xi} \ell^\frac{1}{4} - S^{(6)}(\ell),
\\
&\inf\limits_{\tau\in\left[-\tauo,\tauo\right]} N_0(\ell,\tau) \leqslant a_0(\ell) -\frac{c_1-2\z\left(\frac{3}{2}\right)\xi^\frac{3}{2}}{\pi\xi} \ell^\frac{1}{4} + S^{(6)}(\ell).
\end{align*}
In the same way how inequalities (\ref{4.14a}), (\ref{5.12}) were obtained, by Lemma~\ref{lm3.2} we get:
\begin{align*}
&\frac{T}{\pi}\sqrt{\tht_k^0-\eta_{k+1}^0} +
\frac{T}{2\pi\xi} (\tht_k^0-\eta_{k+1}^0) \geqslant 2\frac{c_1-2\z\left(\frac{3}{2}\right)\xi^\frac{3}{2}}{\pi\xi} \left(\frac{T^2}{\pi^2}\eta_k^0\right)^\frac{1}{4} - 2 S^{(6)}\left(
\frac{T^2}{\pi^2}\eta_k^0\right)-1,
\\
&\tht_k^0-\eta_{k+1}^0 \geqslant \frac{3}{T}\left(c_1-2\z\left(\frac{3}{2}\right)\xi^\frac{3}{2} \right)
\left(\frac{T^2}{\pi^2}\eta_k^0\right)^\frac{1}{4} - \frac{3\pi\xi}{2T} S^{(6)}\left(
\frac{T^2}{\pi^2}\eta_k^0\right)-\frac{3\pi\xi}{2T}-\frac{9\xi^2}{4}.
\end{align*}
Therefore, by inequality (\ref{5.1}) and the minimax principle, we have $\tht_k\geqslant \eta_{k+1}$ once
\begin{equation*}
3\left(c_1-2\z\left(\frac{3}{2}\right)\xi^\frac{3}{2} \right)
\left(\frac{T^2}{\pi^2}\eta_k^0\right)^\frac{1}{4} - \frac{3\pi\xi}{2} S^{(6)}\left(
\frac{T^2}{\pi^2}\eta_k^0\right)- \frac{3\pi\xi}{2}-\frac{9\pi\xi^2 }{4}\geqslant T\om_\cL
\end{equation*}
for all $\eta_k^0\geqslant \tfrac{\pi^2}{T^2}$; the latter condition corresponds to the assumed inequality $\ell\geqslant 1$. Denoting $\ell:=\tfrac{T^2}{\pi^2}\eta_k^0$, we rewrite the above inequality as
\begin{equation*}
\ell^\frac{1}{4} \left(c_1-2\z\left(\frac{3}{2}\right)\xi^\frac{3}{2}-\frac{\pi\xi}{2} \ell^{-\frac{1}{4}} S^{(6)}(\ell)\right) - \frac{\pi\xi}{2} -\frac{3\pi\xi^2}{4}- \frac{T\om_\cL}{4}\geqslant 0
\end{equation*}
and this should hold for all $\ell\geqslant 1$. Explicit formula (\ref{4.16a}) for $S^{(6)}(\ell)$ implies immediately that this inequality is true for all $\ell\geqslant 1$ provided it holds as $\ell=1$. As $\ell=1$, up to obvious transformations, this inequality coincides with condition (\ref{2.10}). This completes the proof of Theorem~\ref{th2.2}.

\section{Discussion of condition (\ref{2.3})}

In this section we discuss the functions $\vp_p(\ell)$ and condition (\ref{2.3}).  Our main conjecture motivated by numerical tests is that condition (\ref{2.3}) holds for all $\xi$ with $\g=0$. The first possible steps in proving this conjecture are as follows.

We begin with a simple bound for $\vp_p(\ell)$. We have
\begin{equation*}
|\vp_p(\ell)|\leqslant \frac{1}{\pi\xi} \sum\limits_{k\in\mathds{Z}} \frac{1}{\left(\frac{k^2}{\xi^2}+p^2\right)^\frac{3}{4}} =\frac{1}{\pi p^\frac{3}{2}\xi} + \frac{2}{\pi} \xi^\frac{1}{2} \sum\limits_{k=1}^{\infty} \frac{1}{\left(\frac{k^2}{\xi^2}+p^2\right)^\frac{3}{4}} .
\end{equation*}
The function $t\mapsto (t^2+p^2\xi^2)^{-\frac{3}{4}}$ decreases monotonically in $t\in[0,+\infty)$ and hence,
\begin{align*}
|\vp_p(\ell)|\leqslant &\frac{1}{\pi p^\frac{3}{2}\xi} + \frac{2}{\pi}\xi^\frac{1}{2} \int\limits_{0}^{+\infty} \frac{\di t}{(t^2+p^2\xi^2)^\frac{3}{4}} = \frac{1}{\pi p^\frac{3}{2}\xi} + \frac{2}{\pi p^\frac{1}{2}} \int\limits_{0}^{+\infty} \frac{\di t}{(t^2+1)^\frac{3}{4}}
\\
=& \frac{1}{\pi p^\frac{3}{2}\xi} + \frac{\Beta(\tfrac{1}{4},\tfrac{1}{2})}{\pi p^\frac{1}{2}},
\end{align*}
where $\Beta(\cdot,\cdot)$ is the Beta function. The obtained estimate yields that as $p\geqslant C_1\ell^\frac{1}{2}$, $C_1=const>0$, we have
\begin{equation*}
|\vp_p(\ell)|\leqslant \frac{1}{\pi C_1^\frac{1}{2}\ell^\frac{1}{4}} \left(\Beta\big(\tfrac{1}{4},\tfrac{1}{2}\big) + \frac{1}{C_1\xi\ell^\frac{1}{2}}\right).
\end{equation*}
Comparing this inequality with condition (\ref{2.3}), we immediately conclude that this condition can be reformulated as
\begin{equation}\label{6.1}
\sup\limits_{p\in\mathds{N},\ p\leqslant C_1\ell^\frac{1}{2}}|\vp_p(\xi,\ell)|\geqslant c_0 \ell^{-\g}.
\end{equation}

In a similar way we can simplify the functions $\vp_p(\ell)$
by replacing them with truncated series. Namely, given $N\in\mathds{N}$, we have
\begin{equation}\label{6.2}
\begin{aligned}
\frac{2}{\pi\xi} \left|\sum\limits_{k=N+1}^{\infty} \frac{\sin\left( 2\pi  \ell^{\frac{1}{2}} \sqrt{\frac{k^2}{\xi^2}+p^2} -\frac{\pi}{4}\right)}{\left(\frac{k^2}{\xi^2}+p^2\right)^\frac{3}{4}}\right| \leqslant &\frac{2\xi^\frac{1}{2}}{\pi} \int\limits_{k=N}^{\infty} \frac{\di t}{(t^2+p^2\xi^2)^\frac{3}{4}}
\\
\leqslant & \frac{2\xi^\frac{1}{2}}{\pi}\int\limits_{N}^{+\infty}\frac{\di t}{t^\frac{3}{2}} = \frac{4\xi^\frac{1}{2}}{\pi N^\frac{1}{2}}.
\end{aligned}
\end{equation}
We fix a constant $C_2>0$ and we truncate the series in (\ref{2.4}):
\begin{equation}\label{6.3}
\Phi_p(\ell):=\frac{1}{\pi\xi}  \sum\limits_{k=-[C_2\ell^\frac{1}{2}]}^{k=[C_2\ell^\frac{1}{2}]}
\frac{\sin\left( 2\pi  \ell^{\frac{1}{2}} \sqrt{\frac{k^2}{\xi^2}+p^2} -\frac{\pi}{4}\right)}{\left(\frac{k^2}{\xi^2}+p^2\right)^\frac{3}{4}}.
\end{equation}
Then by (\ref{6.2}) we get:
\begin{equation*}
|\vp_p(\ell)-\Phi_p(\ell)|\leqslant \frac{4\xi^\frac{1}{2}}{\pi C_2^\frac{1}{2}\ell^\frac{1}{4}}.
\end{equation*}
Hence, we can replace $\vp_p$ by $\Phi_p$ in (\ref{6.1}) and this leads us to an equivalent condition:
\begin{equation}\label{6.4}
\sup\limits_{p\in\mathds{N},\ p\leqslant C_1\ell^\frac{1}{2}}|\Phi_p(\xi,\ell)|\geqslant c_0 \ell^{-\g}.
\end{equation}

Despite the functions $\Phi_p$ are given explicitly by formula (\ref{6.3}), the structure of these functions is quite complicated. As $\ell$ varies, the functions $\Phi_p(\ell)$ oscillate in a non-periodic way having infinitely many zeroes. This non-periodic oscillation is the main obstacle in calculating the supremum in (\ref{6.4}).

A possible way to find such supremum could be to understand the behavior of $\Phi_p(\ell)$ or of $\phi_p(\ell)$ for large $\ell$, that is, the asymptotics as $\ell\to+\infty$. A naive attempt is to replace the series in (\ref{2.4}) by the integral
\begin{equation}\label{6.5}
\int\limits_{\mathds{R}}
\frac{\sin\left( 2\pi  \ell^{\frac{1}{2}} \sqrt{\frac{t^2}{\xi^2}+p^2} -\frac{\pi}{4}\right)\di t}{\left(\frac{t^2}{\xi^2}+p^2\right)^\frac{3}{4}},
\end{equation}
to calculate then the asymptotics of such integral and to try to estimate the error made while passing from the series in (\ref{2.4}) to integral (\ref{6.5}). The asymptotics of the latter integral can be found by the stationary phase method; the leading term is
\begin{equation*}
\int\limits_{\mathds{R}}
\frac{\sin\left( 2\pi  \ell^{\frac{1}{2}} \sqrt{\frac{t^2}{\xi^2}+p^2} -\frac{\pi}{4}\right)\di t}{\left(\frac{t^2}{\xi^2}+p^2\right)^\frac{3}{4}}=\frac{p^\frac{1}{2}}{\pi} \frac{\sin(2\pi p\ell^\frac{1}{2})}{\ell^\frac{1}{4}}+O(\ell^{-\frac{1}{2}}).
\end{equation*}
This leading term decays as $\ell^{-\frac{1}{4}}$. The oscillating part, the function $\ell\mapsto \sin(2\pi p\ell^\frac{1}{2})$, is periodic in $\ell^\frac{1}{2}$. But calculating the functions $\Phi_p(\ell)$ numerically, we see that they do not show such behavior for large $\ell$, namely, these functions do not decay and oscillate non-periodically in $\ell^\frac{1}{2}$. This means that trying to replace the series in (\ref{2.4}) or in (\ref{6.3}) by an integral like (\ref{6.5}) is likely not a proper way in studying the functions $\vp_p$ and $\Phi_p$.

One more property of the functions $\vp_p(\ell)$ is that they solve certain differential equation. We define the function
\begin{equation*}
u=u(l,\mu)=\sum\limits_{k\in\mathds{Z}}\frac{\sin\left( l\sqrt{k^2+\mu} -\frac{\pi}{4}\right)}{(k^2+\mu)^\frac{3}{4}}
\end{equation*}
and we see immediately that
\begin{equation*}
\vp_p(\ell)=\frac{\xi^\frac{1}{2}}{\pi} u\left(\frac{2\pi\ell^\frac{1}{2}}{\xi},p^2\xi^2\right).
\end{equation*}
By straightforward calculations we check that the function $u$ solves the equation
\begin{equation}\label{6.6}
\frac{\p\ }{\p l} \left(\frac{\p^2 u}{\p l\p\mu}+\frac{l}{2}u\right) - \frac{1}{4}u=0,\qquad \mu>0,\quad l\in\mathds{R}.
\end{equation}
We can also write various initial conditions for the function $u$ like
\begin{equation*}
u\big|_{\mu=0}=\sum\limits_{k\in\mathds{Z}} \frac{\sin\left(lk-\frac{\pi}{4}\right)}{k^\frac{3}{2}},
\quad u\big|_{l=0}=-\frac{1}{\sqrt{2}}\sum\limits_{k\in\mathds{Z}} \frac{1}{(k^2+\mu)^\frac{3}{4}}.
\end{equation*}
The issue how to sum these series is open. We can only say that the right hand in the first condition is a $2\pi$-periodic function and
the right hand side in the second condition is a positive monotone  function decaying as $\mu\to+\infty$. But here the main question is how to solve equation (\ref{6.6}) or, at least, how to study the behavior of the solutions for large $l$.

\section*{Acknowledgments}

The author thanks Yu.A.~Kordyukov for very stimulating discussions while working on this paper and an anonymous referee for useful remarks allowed to improve the initial version of the paper.

The reported study was funded by RFBR according to the research project no. 18-01-00046.

\end{document}